\documentclass[11pt]{amsart}

\usepackage[all]{xy}
\usepackage{amsfonts,amssymb,amsmath,amscd}
\usepackage{amsthm}

\usepackage{mathrsfs}

\newcounter{zlist}
\newenvironment{zlist}{\begin{list}{{\rm(\arabic{zlist})}}{
\usecounter{zlist}\leftmargin2.5em\labelwidth2em\labelsep0.5em
\topsep0.6ex\itemsep0.3ex plus0.2ex minus0.3ex
\parsep0.3ex plus0.2ex minus0.1ex}}{\end{list}}

\newcounter{blist}
\newenvironment{blist}{\begin{list}{{\rm(\alph{blist})}}{
\usecounter{blist}\leftmargin2.5em\labelwidth2em\labelsep0.5em
\topsep0.6ex \itemsep0.3ex plus0.2ex minus0.3ex
\parsep0.3ex plus0.2ex minus0.1ex}}{\end{list}}

\newcounter{rlist}
\newenvironment{rlist}{\begin{list}{{\rm(\roman{rlist})}}{
\usecounter{rlist}\leftmargin2.5em\labelwidth2em\labelsep0.5em
\topsep0.6ex\itemsep0.3ex plus0.2ex minus0.3ex
\parsep0.3ex plus0.2ex minus0.1ex}}{\end{list}}

\swapnumbers
\newtheorem{theorem}{Theorem}[section]
\newtheorem{lemma}[theorem]{Lemma}
\newtheorem{thm}[theorem]{}
\newtheorem{proposition}[theorem]{Proposition}

\newtheorem{corollary}[theorem]{Corollary}
\newtheorem{remark}[theorem]{Remark}
\numberwithin{equation}{section}

\newcommand{\A}{{\mathbb {A}}}
\newcommand{\B}{{\mathbb {B}}}
\newcommand{\C}{{\mathbb {C}}}

\newcommand{\bG}{{\mathcal{G}}}
\newcommand{\bH}{{\mathcal{H}}}
\newcommand{\bT}{{\mathcal{T}}}

\newcommand{\cA}{\ensuremath{\mathscr{A}}}
\newcommand{\cC}{\ensuremath{\mathscr{C}}}

\newcommand{\cH}{\ensuremath{\mathscr{H}}}
\newcommand{\cM}{\ensuremath{\mathscr{M}}}
\newcommand{\cN}{\ensuremath{\mathscr{N}}}

\newcommand{\dK}{R}

\newcommand{\lra}{\longrightarrow}
\newcommand{\ve}{\varepsilon}

\newcommand{\kfield}{\mathbf{k}}
\newcommand{\Sy}{{\mathbb {S}}}

\begin{document}

\title{Generalised bialgebras and entwined monads and comonads}

\author{Muriel Livernet}
\address{Universit\'e Paris 13, CNRS, UMR 7539 LAGA, 99 avenue
  Jean-Baptiste Cl\'ement, 93430
  \mbox{Villetaneuse}, France}
\email{livernet@math.univ-paris13.fr}
\author{Bachuki Mesablishvili}
\address{A. Razmadze Mathematical Institute of I. Javakhishvili Tbilisi State University,
6, Tamarashvili Str.,  Tbilisi 0177, Republic of Georgia}
\email{bachi@rmi.ge}
\author{Robert Wisbauer}
\address{Department of Mathematics of HHU, 40225 D\"usseldorf, Germany}
\email{wisbauer@math.uni-duesseldorf.de}

\begin{abstract}
Jean-Louis Loday has defined generalised bialgebras and
proved structure theorems in this setting which can be seen as
general forms of the Poincar\'e-Birkhoff-Witt and the
Cartier-Milnor-Moore theorems.
It was observed by the present authors that parts of the
theory of generalised bialgebras are special cases
of results on entwined monads and comonads and the corresponding
mixed bimodules. In this article the Rigidity Theorem
of Loday is extended to this more general categorical framework.

\end{abstract}

\maketitle

\tableofcontents

\section{Introduction}
The introduction of {\em entwining structures} between an algebra and a
coalgebra by T. Brzezi\'nski and S. Majid in \cite{BrzMaj} opened new perspectives in the
mathematical treatment of  quantum principal bundles.
It turned out that these structures are special cases of distributive laws
treated in Beck's paper \cite{Be}. The latter were also used by Turi and Plotkin
\cite{TuPl} in the context of operational semantics.

These observations led to a revival of the investigation of various forms of
distributive laws. In a series of papers \cite{MW,MW-Bim,MW-Op} it was shown how they
allow for formulating the theory of Hopf algebras and Galois extensions in a general
categorical setting.

On the other hand, {\em generalised bialgebras}  as defined in
Loday \cite[Section 2.1]{Lod},
 are vector spaces which are algebras over an {\em operad ${\cA}$} and coalgebras over a
cooperad $\cC$. Moreover,  the operad $\cA$ and the cooperad $\cC$ are required to be related
by a distributive law. Since any operad $\cA$ yields a monad $\bT_\cA$ and $\cA$-algebras are
nothing else than $\bT_\cA$-modules, and similarly any cooperad $\cC$ yields a comonad $\bG_\cC$
and $\cC$-coalgebras are nothing else than $\bG_\cC$-comodules, generalised bialgebras have
interpretations in terms of bimodules over a bimonad in the sense of \cite{MW-Bim}.

The purpose of the present paper is to make this relationships more precise
(as proposed in \cite[2.3]{MW-Bim}).
We provide a theory for functors on fairly general categories
which leads to the Rigidity Theorem \cite[2.5.1]{Lod}
as a special case. The details of this application are described in Section \ref{Loday}.

\section{Comodules and adjoint functors}

In this section we provide basic notions and properties of
comodule functors and adjoint pairs of functors.
Throughout the paper $\A$ and $\B$ will denote any categories.

\begin{thm} \label{s.s.1.0}{\bf Monads and comonads.}  \em Recall
that a {\em monad} $\bT$ on $\A$ is a
triple $(T,m,e)$ where $T:\A\to \A$ is a functor with natural
transformations $m:TT\to T$, $e:1\to T$ satisfying associativity
and unitality conditions. A $\bT$-{\em module} is an object $a\in
\A$ with a morphism $h:T(a)\to a$ subject to associativity and
unitality conditions. The (Eilenberg-Moore) category of
$\bT$-modules is denoted by $\A_\bT$ and there is a free functor
$$\phi_\bT:\A\to \A_\bT,\; a\mapsto (T(a),m_a),$$
which is left adjoint
to the forgetful functor $$U_\bT:\A_\bT\to \A, \; (a, h) \mapsto a.$$

Dually, a {\em comonad} $\bG$ on $\A$ is a triple $(G,\delta,\ve)$
where $G:\A\to \A$ is a functor with natural transformations
$\delta:G\to GG$, $\ve:G\to 1$, and $\bG$-{\em comodules} are
objects $a\in \A$ with morphisms $\theta:a\to G(a)$. Both notions
are subject to coassociativity and counitality conditions. The
(Eilenberg-Moore) category of $\bG$-comodules is denoted by $\A^\bG$
and there is a cofree functor
$$\phi^\bG:\A\to \A^\bG,\; a\mapsto (G(a),\delta_a),$$
which is right adjoint to the forgetful functor
$$U^\bG:\A^\bG\to \A, \; (a,\theta) \mapsto a.$$
\end{thm}

\begin{thm} \label{s.s.1.1}{\bf $\bG$-comodule functors.}
  {\em For a comonad $\bG=(G,\delta,\varepsilon)$ on
$\A$, a functor $F: \B \to \A$ is a {\em left $\bG$-comodule}
if there exists a natural transformation $\alpha_F : F \to GF$
inducing commutativity of the diagrams

$$\xymatrix{
F \ar@{=}[dr] \ar[r]^-{\alpha_F}& GF \ar[d]^-{\varepsilon F}\\
& F,} \qquad \xymatrix{
F \ar[r]^-{\alpha_F} \ar[d]_-{\alpha_F}& GF \ar[d]^-{\delta F}\\
GF \ar[r]_-{G \alpha_F}& GGF.}
$$

Symmetrically, one defines right $\bG$-comodules.}
\end{thm}

\begin{thm}\label{s.s.1.2}{\bf $\bG$-comodules and adjoint functors.} \em
  Consider a comonad $\bG=(G,\delta,\varepsilon)$ on $\A$
 and an adjunction $F \dashv R:\A \to \B$ with counit $\sigma : FR \to 1$.

There exist bijective correspondences (see \cite{G}) between:
\begin{itemize}
  \item functors $K : \B \to \A^\bG$ with commutative diagrams
$$\xymatrix{ \B  \ar[r]^-{K}\ar[rd]_{F} & \A^\bG \ar[d]^{U^\bG}\\
& \A ;}$$
 \item left $\bG$-comodule structures $\alpha_{F} :F \to GF$ on $F$;

 \item comonad morphisms from the comonad
  generated by the adjunction $F \dashv R$ to the comonad $\bG$;

 \item right $\bG$-comodule structures $\beta_{R} :R \to RG$ on $R$.
\end{itemize}
\medskip

In this case,  $K(b)=(F(b), \alpha_{b})$ for some morphism $\alpha_{b}: F(b) \to GF(b)$,
and the collection $\{\alpha_b,\, b \in\B\}$ constitutes a natural transformation
$\alpha_F:F \to GF$
making $F$ a $\bG$-comodule. Conversely, if $(F,\alpha_F: F\to GF)$ is a $\bG$-module,
then $K: \B \to \A^\bG$ is defined by $K(b)=(F(b), (\alpha_F)_b)$.

\medskip

For any left $\bG$-comodule structure $\alpha_{F} :F \to GF$, the composite
$$\xymatrix{t:FR \ar[r]^-{\alpha_F R}& GFR \ar[r]^-{G \sigma }  & G} $$
is a comonad morphism from the comonad generated by the adjunction $F \dashv R$ to the comonad
$\bG$. Then the corresponding right $\bG$-comodule structure $\beta_{R} :R \to RG$ on $R$ is the composite
$R \xrightarrow{ \eta R}RFR \xrightarrow{R t}RG .$

Conversely, given a  right $\bG$-comodule structure $\beta_{R} :R \to RG$ on $R$, then the comonad
morphism $t:FR \to \bG$ is the composite
$$\xymatrix{FR \ar[r]^-{F \beta_R}& FRG \ar[r]^-{ \sigma G}  & G,}$$ while the corresponding left
$\bG$-comodule structure $\alpha_{F} :F \to GF$ on $F$ is the composite $F \xrightarrow{ F\eta }FRF \xrightarrow{t F}GF.$

\end{thm}

We need the following result, the dual version of Dubuc's theorem \cite{D}.

\begin{thm}\label{Dubuc}{\bf Dubuc's Adjoint Triangle Theorem.}
For categories $\A$, $\B$ and $\C$,
let $F: \A \to \B$ be a functor with right
adjoint $U$ with unit $\eta : 1 \to UF$, and let $K: \C \to \A$ be such that $F'=FK:\C \to \B$ has a right
adjoint with counit $\varepsilon':F'U'\to 1$.
Define
$$\alpha: KU' \xrightarrow{\eta KU'}UFKU'=UF'U' \xrightarrow{U\varepsilon'}U.$$

If $\C$ has equalisers of coreflexive pairs and the functor
$F$ is of descent type, then $K$ has a right adjoint $\dK$ which can be calculated as the equaliser
$$\xymatrix{
\dK\ar[r] & U'F \ar[rr]^{U'F\eta} \ar[dr]|{\eta'U'F}& & U'FUF\\
           & & U'F'U'F=U'FKU'F \ar[ru]|{U'F\alpha F} . } $$
\end{thm}

\smallskip

\begin{thm} \label{s.s.1.3}{\bf Right adjoint of $K$.} \em
Now fix a functor $K : \B \to \A^\bG$ with $U^\bG K=F$ and
suppose that the category $\B$ has equalisers of coreflexive pairs. It then follows from
Theorem \ref{Dubuc} that the functor $K$ has a right adjoint $\overline{R}$
which is determined by the equaliser diagram
\begin{equation}\label{e.1}\xymatrix{
\overline{R} \ar[r]^{i\quad} & R U^\bG
\ar@{->}@<0.5ex>[rr]^-{RU^\bG \eta^\bG} \ar@ {->}@<-0.5ex> [rr]_-{\beta_{R}
U^\bG }&& RGU^\bG=RU^\bG \phi^\bG U^\bG,}\end{equation}
where $\eta^\bG : 1 \to\phi^\bG U^\bG$ is the unit of the adjunction $U^\bG \dashv \phi^\bG.$

An easy inspection shows that the value of $\overline{R}$ at $(a, \theta)\in \A^\bG$ is given by the equaliser diagram
\begin{equation}\label{eq.}\xymatrix{\overline{R}(a, \theta)\ar[r]^-{i_{(a,\theta)}} &R(a) \ar@{->}@<0.5ex>[rr]^-{R(\theta)} \ar@
{->}@<-0.5ex> [rr]_-{(\beta_R)_a}& & RG(a)\,.}\end{equation}
\end{thm}

\begin{theorem} \label{fun.th.} \emph{(}see \cite[Theorem 4.4]{M}\emph{)} A functor $K : \B \to \A^\bG$ with $U^\bG K=F$ is an
equivalence of categories if and only if
\begin{itemize}
  \item [(i)] the functor $F$ is comonadic, and
  \item [(ii)] $t_K$ is an isomorphism of comonads.
\end{itemize}
\end{theorem}

\section{Distributive laws}

Distributive laws were introduced by Beck in \cite{Be}. Here we are mainly interested
in the following case (e.g. \cite{FM} or \cite[5.3]{W-alg}).

\begin{thm} \em\label{f.entw} {\bf Mixed distributive laws.}
Let $\bT =(T,m,e)$ be a monad  and $\bG=(G,\delta,\ve)$ a comonad
on the category $\A$. A natural transformation
$$\lambda: TG\to GT$$ is said to be a
{\em  mixed distributive law} or a {\em (mixed) entwining} provided it induces commutative diagrams
$$
\xymatrix{
   TTG\ar[rr]^{mG} \ar[d]_{T\lambda} & & TG \ar[d]^\lambda \\
   TGT \ar[r]_{\lambda T} & GTT \ar[r]_{Gm} & GT ,}  \quad
     \xymatrix{
   TG\ar[r] ^{T\delta} \ar[d]_{\lambda} &  TGG \ar[r]^{\lambda G} &
     GTG \ar[d]^{G\lambda} \\
   GT \ar[rr]_{\delta T} & & GGT ,}   $$
   $$
    \xymatrix{ G\ar[r]^{eG} \ar[dr]_{Ge} & TG \ar[d]^\lambda \\
          & GT,}\quad
    \xymatrix{ TG\ar[r]^{T\ve} \ar[d]_{\lambda} & T  \\
           GT \ar[ru]_{\ve T}.}$$

Recall (for example, from \cite{W}) that if  $\bT$ is a monad and
$\bG$ is a comonad on a
category $\A$, then the following structures are in bijective
correspondence:
\begin{itemize}
\item mixed distributive laws $\lambda : TG \to GT$;

\item comonads $\widehat{\bG}=(\widehat{G}, \widehat{\delta}, \widehat{\varepsilon})$
on $\A_{\bT}$ that extend $\bG$ in the sense that \\
$U_\bT \widehat{G}= G U_\bT$, $U_\bT
\widehat{\varepsilon}=\varepsilon U_\bT$ and $U_\bT \widehat{\delta}=\delta
U_\bT$;

\item monads $\widehat{\bT}=(\widehat{T}, \widehat{m}, \widehat{e})$ on $\A^{\bG}$
 that extend $\bT$ in the sense that \\
$U^\bG \widehat{\bT}=T U^\bG$, $U^\bG \widehat{e}=e U^\bG$ and $U_\bG \widehat{m}=m U^\bG$.
\end{itemize}

Recall also that
\begin{center}
$\widehat{G}(a,
h)=(G(a), G(h) \cdot \lambda_a)$, $\widehat{\varepsilon}_{(a,
h)}=\varepsilon_a$, $\widehat{\delta}_{(a, h)}=\delta_a$, for
any $(a, h) \in \A_\bT$;
\end{center}
\begin{center}
$\widehat{T}(a, \theta)=(T(a),
\lambda_a \cdot T(\theta))$, $\widehat{e}_{(a, \theta)}=e_a$,
$\widehat{m}_{(a, \theta)}=m_a $ for any $(a, \theta) \in
\A^{\bG}$.
\end{center}

\smallskip

 It follows that for a mixed distributive law $\lambda : TG \to GT$ one may assume
 $$(\A^{\bG})_{\widehat{\bT}}=(\A_{\bT})^{\widehat{\bG}}.$$
We write $\A_{\bT}^{\bG}(\lambda)$ for this category, whose objects,
called $TG$-\emph{bimodules} in \cite{FM},
are triples $(a, h, \theta)$,
where $(a, h)\in \A_\bT$, $(a, \theta) \in \A^{\bG}$ with commuting diagram
$$
\xymatrix{ T(a) \ar[r]^{h} \ar[d]_{T(\theta)} & a \ar[r]^{\theta} & G(a)\\
TG(a) \ar[rr]_{\lambda_a} && GT(a) . \ar[u]_{G(h)} }
$$
Morphisms in this category are morphisms in $\A$ which are
$\bT$-module as well as $\bG$-comodule morphisms.
\end{thm}

\begin{thm} \label{mon-comon}{\bf Entwined monads and comonads.} \em
Let $\bT=(T,m,e)$ be a monad, $\bG=(G,\delta,\ve)$ a comonad on $\A$,
and consider an entwining $\lambda : TG \to GT$ from $\bT$ to $\bG$.
Denote by $\widehat{\bT}=(\widehat{T}, \widehat{m},\widehat{e})$
the monad on $\A^\bG$ lifting  $\bT$
and by $\widehat{\bG}=(\widehat{G},\widehat{\delta},\widehat{\varepsilon})$
 the comonad on $\A_\bT$ lifting $\bG$.

Suppose there exists a functor $K: \A \to (\A_\bT)^{\widehat{\bG}}$
with commutative diagram
\begin{equation}\label{com.triangle}
\xymatrix{ \A  \ar[r]^-{K}\ar[rd]_{\phi_\bT} & (\A_\bT)^{\widehat{\bG}} \ar[d]^{U^{\widehat{\bG}}}\\
& \A_\bT}\end{equation} and consider the corresponding right $\widehat{\bG}$-comodule
structure on $U_\bT$ (see \ref{s.s.1.2})
$$\beta=\beta_{U_\bT} : U_\bT \to U_\bT \widehat{G}=G U_\bT.$$
 Then, for any $(a, h ) \in \A_\bT$, the
$(a,h)$-component $\beta_{(a, h )}=(\beta_{U_\bT})_{(a, h)}$ of $\beta$ is a morphism
$a \to G(a)$ in $\A$. Assuming that $\A$ admits coreflexive equalisers, we obtain by (\ref{eq.}) that
the functor $K$ admits a right adjoint $\dK$ whose value at
$((a,h),\theta)\in (\A_\bT)^{\widehat{\bG}}$ appears as the equaliser
\begin{equation}\label{eq.eq}
\xymatrix{\dK((a,h), \theta)\ar[rr]^-{i_{((a,h),\theta)}} &&a \ar@{->}@<0.5ex>[r]^-{\theta} \ar@
{->}@<-0.5ex> [r]_-{\beta_{(a,h)}}& G(a)\,.}
\end{equation}

Consider now the left $\widehat{\bG}$-comodule structure $\alpha=\alpha_{\phi_\bT} :  \phi_\bT
\to \widehat{G}\phi_\bT$ on $\phi_\bT$  induced by the commutative diagram (\ref{com.triangle}).
As shown in \cite[Theorem 2.4]{MW}, for any $(a,h) \in \A_\bT$, the component
$(t_K)_{(a,h)}$ of the comonad morphism $t_K : \phi_\bT U_\bT \to \widehat{\bG} $,
corresponding to the diagram (\ref{com.triangle}), is the composite
\begin{equation}\label{composite}
T(a) \xrightarrow{T(\beta_{(a, h)})} TG(a)
\xrightarrow{\lambda_a}  GT(a) \xrightarrow{G(h)} G(a)\,.
\end{equation}
\end{thm}

\section{Grouplike morphisms}\label{groupl}

Let $\bG=(G,\delta,\varepsilon)$ be a comonad on a category $\A$. By \cite[Definition 3.1]{MW},
a natural transformation $g:1\to G$ is called a {\em grouplike morphism} provided it is a comonad
morphism from the identity comonad to $\bG$, that is, it induces
commutative diagrams
$$\xymatrix{ 1 \ar[r]^g \ar[dr]_= & G \ar[d]^\varepsilon \\
                & 1 ,}  \quad
  \xymatrix{ 1 \ar[r]^g \ar[dr]_{gg} & G \ar[d]^\delta \\
                & GG .} $$

The dual notion is that of \emph{augmentation}. A monad $\bT$ on $\A$ has an augmentation
if it is endowed with a monad morphism $T \to 1$.

Let $\bT=(T,m,e)$ be a monad and $\bG=(G,\delta,\varepsilon)$ a comonad on
 $\A$ with an entwining  $\lambda:TG\to GT$.
If $\bG$ has a grouplike morphism $g:1\to G$, then the above conditions guarantee
that the morphisms $(g_a:a \to G(a))_{(a,h)\in \A_\bT}$ form the components of a right $\widehat{\bG}$-comodule
structure $\beta=\beta_{U_\bT}:U_\bT \to U_\bT\widehat{G}$ on the functor  $U_\bT:\A_\bT \to \A.$

Observing that in the diagram
$$\xymatrix{ T(a) \ar[rr]^{T(g_a)} \ar[d]_{T(e_a)} && TG(a) \ar[rr]^{\lambda_a} \ar[d]_{TG(e_a)}&& GT(a) \ar[d]_{GT(e_a)} \ar@{=}[rrd] \\
TT(a) \ar[rr]_{T(g_{(T(a)})} && TGT(a) \ar[rr]_{\lambda_{T(a)}}&& GTT(a) \ar[rr]_{G(m_a)}&& GT(a)}$$
\begin{itemize}
  \item the left hand square commutes by naturality of $g$,
  \item the right hand square commutes by naturality of $\lambda$, and
  \item the triangle commutes since $e$ is the unit for the monad $\bT$,
\end{itemize} and recalling that $\alpha$ is the composite $\phi_\bT \xrightarrow{\phi_\bT \eta_\bT}\phi_\bT U_\bT
\phi_\bT \xrightarrow{t_K \phi_\bT} \widehat{G}\phi_\bT$, one concludes by (\ref{composite}) that
\begin{equation}\label{e.2.2}
\text{for every } a\in \A, \,\,\alpha_{a}=\lambda_a \cdot T(g_a).
\end{equation}

This leads to a functor
$$K_g: \A \to (\A_\bT)^{\widehat{\bG}}, \quad a \longmapsto((T(a), m_a), \lambda_a \cdot T(g_a)),
$$
 and the commutative diagram
\begin{equation}\label{comparison}
\xymatrix{
\A \ar[r]^-{K_g} \ar[dr]_{\phi_\bT}& (\A_\bT)^{\widehat{\bG}}
\ar[d]^{U^{\widehat{\bG}}}\\
&\A_\bT .}
\end{equation}

In this case we say that the comparison functor $K_g$ is induced by the grouplike morphism $g:1 \to G$.

Specialising now Theorem \ref{fun.th.} to the present situation gives

\begin{theorem} \label{fun.th.1} Let $\bT=(T,m,e)$ be a monad and $\bG=(G,\delta,\varepsilon)$ a comonad on
$\A$ with an entwining  $\lambda:TG\to GT$. If $g:1\to G$ is a grouplike morphism of the comonad $\bG$,
then the induced  functor $K_g:\A \to (\A_\bT)^{\widehat{\bG}}$ is an equivalence of
categories if and only if
\begin{itemize}
  \item [(i)] the functor $\phi_\bT$ is comonadic, and
  \item [(ii)] the composite \begin{equation}\label{composite.1}
T(a) \xrightarrow{T(g_a)} TG(a)
\xrightarrow{\lambda_a}  GT(a) \xrightarrow{G(h)} G(a)
\end{equation} is an isomorphism for every $(a,h) \in \A_\bT$.
\end{itemize}
\end{theorem}

\begin{thm} \label{rem.1}{\bf Remark.} \em
It follows from \cite[Theorem 2.12]{MW-Op} that the second
condition of Theorem \ref{fun.th.1} is equivalent to saying that the composite
$$TT(a) \xrightarrow{T(g_{T(a)})} TGT(a)
\xrightarrow{\lambda_{T(a)}}  GTT(a) \xrightarrow{G(m_a)} GT(a)$$ is an isomorphism
for every $a\in \A$.
\end{thm}

\section{Compatible entwinings}\label{Compat}

Let $\underline{\bH}=(H, m, e)$ be a monad, $\overline{\bH}=(H, \delta,
\varepsilon)$ a comonad on $\A$, and let $\lambda : HH\to HH$
be an entwining from the monad $\underline{\bH}$ to the comonad $\overline{\bH}$.
The datum $(\underline{\bH},\overline{\bH},\lambda)$ is called a {\em monad-comonad
triple}. The objects of the category $\A^{\overline{\bH}}_{\underline{\bH}}(\lambda)$
are called {\em (mixed) $\lambda$-bimodules}.

\begin{lemma}\label{L.1}
The triple $(H(a), m_a, \delta_a)$ is a $\lambda$-bimodule
for all $a \in \A$ if and only if we have a commutative diagram
\begin{equation}\label{e.3}
\xymatrix{ HH \ar[r]^m \ar[d]_{H\delta} & H \ar[r]^\delta &HH\\
HHH \ar[rr]_{\lambda H} & & HHH \ar[u]_{Hm}. }
\end{equation}
In this case, there are functors
\begin{zlist}
\item $K:\A \to (\A_{\underline{\bH}})^{\widehat{\overline{\bH}}}, \;
a \longmapsto ((H(a),m_a), \delta_a),$
satisfying $\phi_{\underline{\bH}}=U^{\widehat{\overline{\bH}}}K$;

\item $K':\A \to (\A^{\overline{\bH}})_{\widehat{\underline{\bH}}}, \;
 a \longmapsto ((H(a),\delta_a), m_a)$,
 satisfying $ U_{\widehat{\underline{\bH}}} K'=\phi^{\overline{\bH}}$.
\end{zlist}
\end{lemma}

\begin{thm}\label{def-comp}{\bf Definitions.} \em
Given a monad-comonad triple $(\underline{\bH},\overline{\bH},\lambda)$,
the entwining $\lambda : HH\to HH$
is said to be {\em compatible} provided Diagram (\ref{e.3}) is commutative;
 then $(\underline{\bH},\overline{\bH},\lambda)$ is said to be a {\em compatible monad-comonad triple}.

The triple $(\underline{\bH},\overline{\bH},\lambda)$ is called a {\em bimonad}
if it is a compatible triple with additional commutative diagrams
(see \cite[Definition 4.1]{MW-Bim})
\begin{equation}\label{D.1.18b}
 \xymatrix{ HH \ar@{}[rd]|{(i)}\ar[r]^{H\varepsilon} \ar[d]_m & H  \ar[d]^\varepsilon \\
 H \ar[r]^\varepsilon & 1 , } \qquad
\xymatrix{
1 \ar@{}[rd]|{(ii)} \ar[r]^-{e} \ar[d]_-{e} & H \ar[d]^-{\delta}\\
H \ar[r]_-{He}& HH,}  \qquad
\xymatrix{ 1 \ar[r]^e \ar@/_1pc/ [dr]_= & H \ar@{}[ld]|(0.45){(iii)}\ar[d]^\ve\\
         & 1 .}
\end{equation}

Notice that for any monad-comonad triple $(\underline{\bH},\overline{\bH},\lambda)$,
to say that Diagram (\ref{D.1.18b})(i) commutes is to say that $\ve: H \to 1$
is an augmentation of the monad $\underline{\bH}$\,, while to say that Diagram (\ref{D.1.18b})(ii)
commutes is to say that $e:1 \to H$ is a grouplike morphism of the comonad $\overline{\bH}$.
Thus, for any bimonad $(\underline{\bH},\overline{\bH},\lambda)$,
$e$ is a grouplike morphism of the comonad $\overline{\bH}$ and $\ve$ is an augmentation
of the monad $\underline{\bH}$.
\end{thm}

\begin{proposition} \label{Hopf.1}
Let $(\underline{\bH},\overline{\bH},\lambda)$ be a compatible monad-comonad triple.
If $\delta \cdot e =He \cdot e$ (i.e. $e:1 \to H$ is a grouplike morphism of $\overline{\bH}$),
then $\delta=\lambda\cdot He$ and the comparison functor $K$ in Lemma \ref{e.3} is induced by
the grouplike morphism $e$, that is  $K=K_e$.
\end{proposition}

\begin{proof} Assume that $\delta \cdot e =He \cdot e$ and that $\lambda$ is compatible.
 Then, in the diagram
$$
\xymatrix{H \ar[r]^-{He} \ar[d]_{He}& HH \ar[r]^m \ar[d]^{H\delta} & H \ar[r]^\delta &HH\\
HH \ar[rrd]_{\lambda}\ar[r]^{HHe} &HHH \ar[rr]^{\lambda H} & & HHH \ar[u]_{Hm}\\
&&HH \ar[ru]_{HHe}& ,}
$$
the rectangles commute. Since the triangle is also commutative by naturality of composition
and since $m\cdot He=1$, it follows that $\delta=\lambda\cdot He$.
From Section \ref{groupl} and (\ref{e.2.2}), we conclude that the comparison functor $K$  is induced by
the grouplike morphism $e$, that is  $K=K_e$.
\end{proof}

\begin{remark}\label{rem:m} \em Note that if $\ve\cdot m=\ve\cdot H\ve$
(i.e. $\ve:H \to 1$ is an
augmentation of $\underline{\bH}$) and $\lambda$ is compatible, then  postcomposing the
diagram (\ref{e.3}) with the morphism $H\ve$ implies $m=H\ve\cdot\lambda$.
\end{remark}

In the next propositions we do not require a priori $\lambda$ to be a
compatible entwining.

\begin{proposition} \label{conditions}
Let $(\underline{\bH},\overline{\bH},\lambda)$ be a monad-comonad triple.
\begin{rlist}
    \item  If $\,\,\delta=\lambda \cdot He$, then $\delta \cdot e=He \cdot e$;
    \item  if $\,\,m=H\varepsilon \cdot \lambda$, then
       $\varepsilon \cdot m = \varepsilon \cdot H\varepsilon$.
\end{rlist}
Moreover, if one of these conditions is satisfied,
then $\varepsilon \cdot e=1$, provided that $e: 1 \to H$ is a
(componentwise) monomorphism or $\ve$ is a (componentwise) epimorphism.
\end{proposition}

\begin{proof}(i) Assume $\delta=\lambda \cdot He$. Since
$He \cdot e=eH \cdot e$ (by naturality) and $\lambda \cdot eH=He$
(see \ref{f.entw}),
$$\delta\cdot e=\lambda \cdot He \cdot e=\lambda \cdot eH \cdot e=He
\cdot e.$$

(ii) Assume $m=H\varepsilon \cdot \lambda$. Since $\varepsilon
\cdot H\varepsilon =\varepsilon \cdot \varepsilon H$ and
$\varepsilon H \cdot \lambda=H \varepsilon$ (see \ref{f.entw}),
$$\varepsilon\cdot m=\varepsilon \cdot H\varepsilon \cdot \lambda=
\varepsilon \cdot H\varepsilon \cdot \lambda =\varepsilon \cdot H
\varepsilon.$$

To show the final claim, observe that   $\delta=\lambda \cdot He$ implies
 $$1=\varepsilon H \cdot \delta =\varepsilon H \cdot \lambda \cdot He
=H\varepsilon \cdot He ,$$
and  $\,\,m=H\varepsilon \cdot\lambda$ implies
 $$1=m \cdot eH= H\varepsilon \cdot \lambda \cdot eH =H\varepsilon \cdot He,$$
so in both cases, $1=H\varepsilon \cdot H e$.
Naturality of $e$ and $\ve$ imply commutativity of the diagrams, respectively,
$$\xymatrix{H \ar[r]^{eH}\ar[d]_{\varepsilon}& HH
\ar[d]^{H\varepsilon}\\
1 \ar[r]_e &H , }\quad
\xymatrix{H \ar[r]^{He}\ar[d]_{\varepsilon}& HH
\ar[d]^{\varepsilon H}\\
1 \ar[r]_e &H .}
$$
From the left hand diagram one gets
$$e=H\varepsilon \cdot He \cdot
e=H\varepsilon \cdot eH \cdot e=e \cdot \varepsilon \cdot e,$$
thus if $e$ is a (componentwise) monomorphism, $\varepsilon\cdot e=1$,
while the right hand diagram implies
$$\ve = \ve\cdot H\ve \cdot He = \ve \cdot e\cdot \ve $$
and hence $\varepsilon\cdot e=1$ provided $\ve$ is  a (componentwise) epimorphism.
\end{proof}

\begin{thm}\label{liv}{\bf Lemma.}
Let $(\underline{\bH},\overline{\bH},\lambda)$ be a monad-comonad triple.
  If
\begin{center}
 $m=H\varepsilon \cdot \lambda$ \; or \; $\delta=  \lambda \cdot He$,
\end{center}
then $\lambda$ is compatible, that is, diagram (\ref{e.3}) is commutative.
\end{thm}
\begin{proof}
If $\delta=  \lambda \cdot He$, the triangle is commutative in the diagram
 $$\xymatrix{
 HH \ar[r]^{H\delta\quad} \ar[d]_{HHe} & HHH \ar[r]^{\lambda H} & HHH\ar[d]^{Hm} \\
 HHH \ar[ru]_{H\lambda} \ar[r]_{mH} & HH \ar[r]^\lambda & HH ,} $$
whereas the trapezium is commutative by the entwining property of $\lambda$.
The left path of the outer diagram is
$$ \lambda \cdot mH \cdot HHe = \lambda \cdot He \cdot m = \delta\cdot m.$$
This shows that (\ref{e.3}) is commutative.

In a similar way  the claim for $m=H\varepsilon \cdot \lambda$ is proved.
\end{proof}

To sum up, combining Proposition \ref{conditions}, Remark \ref{rem:m} and Lemma \ref{liv} yields

\begin{proposition} \label{bimonad}
Let $(\underline{\bH},\overline{\bH},\lambda)$ be a monad-comonad triple.

\begin{zlist}
\item  $\delta=\lambda \cdot He$ if and only if $\lambda$ is compatible and
 $\delta\cdot e= He\cdot e$;

\item $m=H\varepsilon \cdot \lambda$ if and only if $\lambda$ is compatible
    and $\ve\cdot m= \ve\cdot H\ve$;

\item if $\delta=\lambda \cdot He$, $m=H\varepsilon \cdot \lambda$, and
$e: 1 \to H$ is a (componentwise) monomorphism or $\varepsilon$ is a (componentwise)
epimorphism,
then  $(\underline{\bH},\overline{\bH},\lambda)$ is a bimonad (see {\rm \ref{def-comp}}).
\end{zlist}
\end{proposition}

\medskip

If $(\underline{\bH},\overline{\bH},\lambda)$ is a monad-comonad triple such that
$\delta=  \lambda \cdot He$, then
$(\underline{\bH},\overline{\bH},\lambda)$ is a compatible monad-comonad triple by Lemma \ref{liv},
and hence, by Proposition \ref{Hopf.1}, the assignment $a \longmapsto (H(a), m_a, \delta_a)$ yields
the functor $K_e:\A \to \A^{\overline{\bH}}_{\underline{\bH}}(\lambda)$
with commutative diagram
$$\xymatrix{
\A \ar[r]^-{K_e} \ar[dr]_{\phi_{\underline{\bH}}}& \A^{\overline{\bH}}_{\underline{\bH}}(\lambda)=(\A_{\underline{\bH}})^{\widehat{\overline{\bH}}}
\ar[d]^{U^{\widehat{\overline{\bH}}}}\\
&\A_{\underline{\bH}} \,.}$$

Recall from \cite{MW-Bim} that a bimonad $\bH$ is said to be a \emph{Hopf monad} provided it has an \emph{antipode}, i.e.
there exists a natural transformation $S: H \to H$ such that $m\cdot HS \cdot \delta=e \cdot \ve=m\cdot SH \cdot \delta.$

\begin{theorem} \label{identity}
Let $(\underline{\bH},\overline{\bH},\lambda)$ be a monad-comonad triple on a
Cauchy complete category $\A$. Assume that $\delta=  \lambda \cdot He$ and $e: 1 \to H$
is a (componentwise) monomorphism. Then the following are equivalent:
\begin{blist}
\item $K_e:\A \to \A^{\overline{\bH}}_{\underline{\bH}}(\lambda)$ is an equivalence of
 categories;
\item the composite $H(a) \xrightarrow{\delta_a}  HH(a) \xrightarrow{H(h)} H(a)$
is an isomorphism for every $(a,h) \in \A_{\underline{\bH}}$;
\item the composite $HH \xrightarrow{\delta H}  HHH \xrightarrow{Hm} HH$ is an isomorphism.
\end{blist}
If, in addition, $\ve: H \to 1$ is an augmentation of the monad $\underline{\bH}$,
then $\bH$ is
a Hopf monad.
\end{theorem}

\begin{proof} Since $\delta=  \lambda \cdot He$, (a) $\Rightarrow$ (b) is trivial by Theorem
\ref{fun.th.1}, while (b) and (c) are equivalent by Remark \ref{rem.1}.

Given (c), it follows from Theorem \ref{fun.th.1} that
$K$ is an equivalence of categories if and only if the functor $\phi_{\underline{\bH}}$ is comonadic.
But by \cite [Corollary 3.17]{Me} this is always the case,
since $e: 1 \to H$ is a monomorphism and hence $\varepsilon \cdot e=1$ by Proposition \ref{conditions}.
This proves the implication (c)$\Rightarrow$(a).

Finally, if $\ve: H \to 1$ is an augmentation of the monad $\underline{\bH}$, then $\ve \cdot m=\ve \cdot H \ve$,
and since $(\underline{\bH},\overline{\bH},\lambda)$ is compatible, $m=H\ve \cdot \lambda$ by Proposition \ref{bimonad}.
Since $\delta=  \lambda \cdot He$, $\delta \cdot e=He \cdot e$ again by Proposition \ref{bimonad}. Thus $\bH$ is
a bimonad and it now follows from \cite[3.1]{MW-Op} that $\bH$ is a Hopf monad.
\end{proof}

\section{Generalised bialgebras} \label{Loday}

In this section, we apply our results in the context of operads to recover
results of Loday on generalised bialgebras in  \cite{Lod}. The Leitmotiv of the section is that a
(co)operad is a particular type of (co)monad. Let $\kfield$ denote a field and
$\A$ the category of $\kfield$-vector spaces.

\begin{thm} {\bf Schur functors.}\label{schurfunctors}{\em \
An {\it $\Sy$-module} $\cM$ in $\A$ (or {\it vector species}) is a  collection of objects $\cM(n)$, for $n>0$,
together with an action of the symmetric group $S_n$. To an $\Sy$-module $\cM$ one associates the functor
$$\begin{array}{cccc}
F_\cM:&\A&\lra & \A\\
&V&\mapsto& \bigoplus\limits_{n>0}\cM(n)\otimes_{\kfield[S_n]} V^{\otimes n}\end{array}.$$
Such a functor is called a {\it Schur functor}. Joyal proved in \cite{Joy} that for
two $\Sy$-modules
$\cM$ and $\cN$, the composite $F_\cM\cdot F_\cN$ is  a Schur functor of the form
$F_{\cM\circ \cN}$ with $\cM\circ \cN$ being the $\Sy$-module defined by
$$(\cM\circ \cN)(n)= \bigoplus\limits_{k>0,i_1+\ldots+i_k=n}\cM(k)\otimes_{\kfield[S_k]}\hbox{\rm Ind}_{S_{i_1}\times\ldots\times S_{i_k}}^{S_n}
\cN(i_1)\otimes\ldots\otimes \cN(i_k).$$
The product $\circ$ is called the \emph{plethysm} of $\Sy$-modules, and the category of $\Sy$-modules, together with the plethysm is a monoidal category. The unit for the plethysm is the $\Sy$-module
$$1(n)=\begin{cases} \kfield, & \text{ if } n=1, \\ 0, & \text{ else. }\end{cases}$$

For our purpose, we will always assume that any $\Sy$-module $\cM$ satisfies
$\cM(1)=\kfield$.

We denote by $e_\cM:1\to F_\cM$ the natural transformation which maps $V$  
to the summand $V$ of $F_\cM(V)$
and by $\ve_\cM:F_\cM\to 1$ the projection of $F_\cM(V)$  
to the summand $V$. Then $\ve_\cM \cdot e_\cM=1$.
}
\end{thm}

\begin{thm} {\bf Operads, cooperads.}{\em \
An {\it operad} $\cA$ in $\A$ is a monoid in the monoidal category of $\Sy$-modules.
This amounts to say that the functor $F_{\cA}$
is the functor part of a monad $\bT_{\cA}=(F_{\cA},m_\cA,e_\cA)$.

An {\em algebra  over an operad $\cA$}, or {\em $\cA$-algebra}, is a $\bT_{\cA}$-module.
 Hence, the free
$\cA$-algebra generated by a vector space $V$ is nothing else than
$(T_{\cA}(V),{(m_\cA)}_V, (e_\cA)_V)$.

A {\em cooperad $\cC$} in $\A$ is a comonoid in the monoidal category  of $\Sy$-modules.
This amounts to say that the functor  $F_{\cC}$ is the functor part of a comonad
$\bG_{\cC}=(F_{\cC},\delta_\cC,\ve_\cC).$

A {\em coalgebra over a cooperad $\cC$}, or {\em $\cC$-coalgebra},
is a $\bG_{\cC}$-comodule.

Note that one has to be a little careful with the definition of cooperads if one wants a
linear duality between operads and cooperads (see \cite{LV}).
With our definition and assumptions, any  coalgebra over a cooperad $\cC$  is naturally conilpotent.

We assume that, for any $\Sy$-module $\cM$, the $\kfield$-vector
space  $\cM(n)$ is finite dimensional. We assume also that either the action of the symmetric group is free
or the field $\kfield$ has characteristic $0$.}
\end{thm}

\begin{proposition}\label{op:compatible} If $\cA$ is an operad, then $\varepsilon_{\cA}$ is
an augmentation for the monad $\bT_{\cA}$.
If $\cC$ is a cooperad then $e_\cC$ is a grouplike morphism for the comonad $\bG_{\cC}$.
\end{proposition}

\begin{proof}  The unit for the plethysm
forms a (co)operad and the associated (co)monad is the identity functor.
Let $m:\cA\circ\cA\rightarrow \cA$ denote the operad composition.
One has to prove that, for every $n\geq 1$, the following diagram is commutative:
$$ \xymatrix{(\cA\circ\cA)(n)\ar[r]^>>>>>{\cA\circ\varepsilon_\cA} \ar[d]_m &
\cA(n) \ar[d]^{\varepsilon_{\cA}} \\
 \cA(n) \ar[r]^{\varepsilon_{\cA}} & 1(n) . }
 $$
If $n>1$, then the diagram commutes because the top and bottom compositions vanish.
If $n=1$, since $\cA(1)=\kfield$ then
$(\cA\circ\cA)(1)=\kfield\otimes\kfield=\kfield$ and $m$ is the identity as well as
$\cA\circ\varepsilon_\cA$ and $\varepsilon_\cA$. So the diagram is commutative.
Furthermore, we have seen in Section \ref{schurfunctors} that
$\ve_\cA\cdot e_\cA=1$.
A similar proof shows that $e_\cC$ is a grouplike morphism for the comonad $\bG_{\cC}$.
\end{proof}

\begin{thm} {\bf Distributive laws and generalised bialgebras.} {\em \
Let $\cA$ be an operad and $\cC$ be a cooperad.

\textbf{(H0)} A distributive law between $\cA$ and $\cC$ is a morphism of $\Sy$-modules
$\cA\circ \cC\rightarrow \cC\circ \cA$
satisfying some relations which amount to say that the corresponding natural
transformation
$$\lambda:F_{\cA\circ\cC}=F_\cA F_\cC\lra F_\cC F_\cA=F_{\cC\circ\cA}$$
 is an entwining.   If such an entwining exists,
we say, as in \cite{Lod}, that {\it hypothesis} (H0) \emph{is satisfied}. Under this hypothesis,
an object $(V,h,\theta)$ in $(\A_{\bT_\cA})^{{\widehat \bG_\cC}}$ is called a {\it $(\cC,\cA)$-bialgebra}.
\medskip

\textbf{(H1)} Assume that there is  a map
$\alpha:\cA\rightarrow \cC\circ \cA$ making $\cA$ a left $\cC$-co\-module, that is,
every
free $\cA$-algebra is endowed with a structure of a $\cC$-coalgebra.
This amounts to say that there is a
functor $K:\A\lra(\A_{\bT_\cA})^{{\widehat \bG_\cC}}$ such that the diagram
(\ref{com.triangle}) is commutative.
If such a functor exists, we say, as in \cite{Lod},
that {\it hypothesis} (H1) \emph{is satisfied}. The
corresponding left $\bG_\cC$-comodule structure on $\bT_\cA$ is given by
$\alpha: F_\cA\rightarrow F_\cC F_\cA$.

At the level of $\Sy$-modules one gets that $\alpha_1:\cA(1)=\kfield\rightarrow (\cC\circ\cA)(1)=\kfield$
is the identity, because $(\ve_\cC\circ\cA)\cdot\alpha=1$, so that
$$\alpha\cdot e_\cA=e_\cC F_\cA\cdot e_\cA=F_\cC e_\cA\cdot e_\cC.$$
Thus, the diagram
\begin{equation}\label{aug}
\xymatrix @C=1.1in{
&F_\cA\ar[rd]^{\alpha} &\\
1 \ar[ru]^{e_\cA} \ar@{->}[r]_(0.65){e_\cA} \ar[rd]_{e_\cC}& F_\cA \ar@{->}[r]_(0.35){e_\cC T}& F_\cC F_\cA\\
&F_\cC \ar[ru]_{F_\cC e_\cA}&}
\end{equation} commutes. As a consequence, if $(V,h) \in \A_{\bT_\cA}$, then
$$F_\cC(h)\cdot  \alpha_V \cdot (e_\cA)_V=F_\cC(h) \cdot(F_\cC e_\cA)_V\cdot (e_\cC)_V=(e_\cC)_V,$$
and since the $(V,h)$-component $\beta_{(V,h)}$ of the right $\bG_\cC$-comodule structure on
$\beta:U_{\bT_\cA} \to U_{\bT_\cA}\widehat{F_\cC}$ is just the composite $F_\cC (h)\cdot  \alpha_V \cdot (e_\cA)_V$,
we get
\begin{equation}\label{Reduction}
\beta_{(V,h)}=(e_\cC)_V.
\end{equation} Thus, $\beta$ is defined by the grouplike morphism $e_\cC:1\rightarrow F_\cC$
and hence the comparison functor $K:\A\lra(\A_{\bT_\cA})^{{\widehat \bG_\cC}}$ is induced by this grouplike morphism, i.e. $K=K_{e_\cC}$.
So we can apply the results of the previous sections to the present setting, in particular,
(\ref{e.2.2}) gives
\begin{equation}\label{Reduction.1}
\alpha=\lambda \cdot F_\cA e_\cC.
\end{equation}

\medskip

We assume that the hypotheses (H0) and (H1) hold. Consider the $\cC$-comodule map
$\varphi:\cA\rightarrow \cC$ induced by the projection $\ve_\cA: \cA \rightarrow 1$.
Since $\varphi=(\cC \circ \ve_\cA)\cdot \alpha$, where $\alpha :\cA \rightarrow \cC\circ \cA$
is the $\cC$-comodule morphism of hypothesis (H1), one has, for every $\mu\in\cA(n)$,
\begin{equation}\label{E:alpha}
\alpha(\mu)=\varphi_n(\mu)\otimes 1^{\otimes n}+\sum\limits_{k<n} c_k^\mu\otimes \alpha_1^\mu\otimes\ldots\otimes \alpha_k^\mu,
\end{equation}
where $\varphi_n$ is the component of $\varphi$
on $\cA(n)$, $c_k^\mu\in\cC(k)$, and $\alpha_i^\mu\in\cA(l_i)$ with $\sum_{i=1}^k l_i=n$.
\smallskip

\textbf{(H2iso)} When $\varphi$ is an isomorphism, we say, as in \cite{Lod},
that {\it hypothesis} (H2iso) \emph{is satisfied}.

\medskip
In the sequel we will be interested in the link between $\varphi$ and the comonad
morphism
$t:\phi_{\bT_\cA} U_{\bT_\cA}\lra \widehat F_\cC$ as in section \ref{s.s.1.2}.
Recall that for every
$(V,h)\in\A_{\bT_\cA}$, $t_{(V,h)}$ is the composite
$$
F_\cA(V)\xrightarrow{\alpha_V}  F_\cC F_\cA(V)\xrightarrow{F_\cC h}  F_\cC (V).
$$

\begin{lemma}\label{L:H2} Assume the hypotheses {\rm (H0)} and {\rm (H1)}. Then
the map $\varphi$ is an isomorphism if and only if $t$ is an isomorphism.
\end{lemma}

\begin{proof} We use the natural arity-grading on $\Sy$-modules.
Given $\mu\in\cA(n), \underline v\in V^{\otimes n}$, one has
$$t_{(V,h)} (\mu\otimes\underline v)=
 \varphi_n(\mu)\otimes\underline v+
 \sum_{k<n}c_k^\mu\otimes \alpha_1^\mu(\underline v_1)\otimes\ldots\otimes \alpha_k^\mu(\underline v_k),$$
where $\underline v=\underline v_1\otimes\ldots\otimes \underline v_k\in V^{\otimes l_1}\otimes\ldots\otimes V^{\otimes l_k}$.
This is a triangular system with dominant coefficient $\varphi_n$.
As a consequence, we get that
if $\varphi$ is an isomorphism so is $t_{(V,h)}$.
The converse is immediate because $\varphi_V=t_{(V,(\ve_\cA)_V)}$ for all $V \in \A$.
\end{proof}}
\end{thm}

\begin{thm} {\bf The primitive part of a $(\cC,\cA)$-bialgebra.}{\em \
Because the category of $\kfield$-vector spaces admits equalisers, under the hypotheses
(H0) and (H1), the functor $K$ admits a right adjoint $\dK$ whose value
at $((H,h),\theta)\in (\A_{\bT_\cA})^{\widehat{\bG_\cC}}$
appears as the equaliser
$$\xymatrix{\dK((H,h), \theta)\ar[rr]^-{i_{((H,h),\theta)}} &&H \ar@{->}@<0.5ex>[r]^-{\theta} \ar@
{->}@<-0.5ex> [r]_-{(e_\cC)_{H}}& \cC(H)\,.}$$
As a consequence,
$$\dK((H,h), \theta)=\{x\in H,\theta(x)=1\otimes x\},$$
and thus $\dK((H,h), \theta)$ is just the {\it primitive part} ${\hbox{\rm Prim}}V$ of the
$(\cC,\cA)$-bialgebra $(H,h,\theta)$ in the sense of Loday \cite{Lod}.}
\end{thm}

\smallskip

We are now in the position to state and prove our main result.

\begin{thm} {\bf Rigidity Theorem.}\label{mein.th} \emph{(\cite[Theorem 2.3.7]{Lod})}
Let $\cA$ be an operad,
$\cC$ a cooperad, and $\bT_\cA=(F_\cA,m,e_\cA)$ and $\bG_\cC=(F_\cC,\delta,\varepsilon_\cC)$ the corresponding monad and comonad on $\A$.
Suppose that the hypotheses
{\rm (H0), (H1)} and {\rm (H2iso)} are fulfilled. Then the comparison functor
$$K_{e_\cC}:\A\lra (\A_{\bT_\cA})^{\widehat{\bG_\cC}}$$
is an equivalence of categories. Hence, in particular, any
$(\cC,\cA)$-bialgebra $(H,h,\theta)$ is a free $\cA$-algebra
and a cofree conilpotent $\cC$-coalgebra generated by $\hbox{\rm Prim}H$.
\end{thm}

\begin{proof} Because the hypothesis (H2iso) is satisfied, it follows from Lemma \ref{L:H2} that
$t_{(V,h)}$ is an isomorphism for all $(V,h)\in \A_{F_{\cA}}$.
Moreover, since $\ve_\cA \cdot e_\cA=1$,
and since $\A$ is clearly Cauchy complete, the functor $\phi_{\bT_\cA} :\A \to \A_{\bT_\cA}$ is comonadic by
\cite[Corollary 3.17]{Me}.
Applying now Theorem \ref{fun.th.1}, we get the result.
\end{proof}

\begin{thm} {\bf Remark.}{\em \ In \cite{Lod}, for the proof of this theorem,
Loday builds idempotents to produce a projection onto the primitive part.
An advantage of our proof is that it does not need such a construction.}
\end{thm}

The following corollary is a special case of the Rigidity Theorem, where it is not necessary to verify hypothesis (H2iso).

\begin{corollary}\label{bialgebra}
Let $\cM$ be an $\Sy$-module carrying a structure of an operad $\cA=(\cM,m,e_\cM)$,
a structure of cooperad $\cC=(\cM,\delta,\varepsilon_\cM)$, and let
\begin{center}
$\lambda : \cM\circ \cM\to \cM\circ\cM$
\end{center}
be an entwining between $\cA$ and $\cC$.
If one of the three equivalent conditions
$$\begin{array}{ccc}
{\rm (i) } \; \lambda \text{ is compatible},& {\rm(ii) }\;
\delta=\lambda\cdot (\cM\circ e_\cM),& {\rm (iii) }\;
 m=(\cM\circ\varepsilon_\cM)\cdot\lambda,
\end{array}$$
holds, then the compatible monad-comonad triple $(\bT_\cA,\bG_\cC,\lambda)$ is a
Hopf monad. Moreover,
any $(\cC,\cA)$-bialgebra  is a free $\cA$-algebra
and a cofree conilpotent $\cC$-coalgebra.
\end{corollary}

\begin{proof}Let us denote by $\cH$ the monad-comonad triple $(\bT_\cA,\bG_\cC,\lambda)$.
By Proposition \ref{op:compatible}, the triple satisfies
Relations (\ref{D.1.18b}), and since $e_\cM$ is a componentwise monomorphism, $\cH$ is a
bimonad by Proposition \ref{bimonad}.
Thus there is a comparison
functor
$$K:\A \to (\A_{\bT_\cA})^{\widehat{\bG_\cC}}, \;
V \longmapsto ((\cM(V),m_V), \delta_V),$$
and $K=K_{e_\cM}$.

We can apply Theorem \ref{identity}
to conclude that the functor $K$ is an equivalence of categories if and only if
the composite
$$\cM(V) \xrightarrow{\delta_V}  (\cM\circ\cM)(V) \xrightarrow{\cM(h)} \cM(V)$$ is
an isomorphism for every $(V,h) \in
\A_{\bT_\cA}$. But  $\cM(h) \cdot \delta_V=t_{(V,h)}$, where $t_{(V,h)}$ is  the
$(V,h)$-component of the comonad morphism
$t:\phi_{\bT_\cA} U_{\bT_\cA} \to \widehat{\bG_\cC}$
induced by $K$. It follows that $\varphi_V=t_{(V,\,\ve_V)}=\cM(\ve_V) \cdot \delta_V=1$
for every
$V \in \A$. Thus $\varphi$ is an isomorphism and then $t$ is also an isomorphism by
Lemma \ref{L:H2}.
Hence $K$ is an equivalence of categories. It now follows from \cite[3.1]{MW-Op}
that $\bH$ is a Hopf monad.
Furthermore, the Rigidity Theorem applies to our case because (H2iso) is satisfied.
\end{proof}

\begin{thm} {\bf Example.} {\em As an example we treat the case of infinitesimal bialgebras. Consider the
functor $V\mapsto \cA(V)=\oplus_n V^{\otimes n}$. It forms a monad $\bT=(\cA,m,e)$ for the concatenation
product. One can formulate this as

$$\begin{array}{ccccccccc}
m_V:& \cA_1\cA_2(V) & \rightarrow & \cA (V) & & e_V:& V & \rightarrow & \cA (V)\\
&\otimes_1 & \mapsto & \otimes & & & v &\mapsto &  v\\
&\otimes_2 & \mapsto & \otimes &&&&& \\
& v & \mapsto &v  &&&&&
\end{array}$$
where $\cA_1$ denotes the ``first copy" of $\cA$.
It reads like this:  any word in $\cA_1\cA_2(V)$ is composed with letters in $\{\otimes_1,\otimes_2,v\in V\}$
and the map indicates how it acts on letters.

The functor $\cA$ forms a comonad $\bG=(\cA,\delta,\ve)$ with the deconcatenation
$$\begin{array}{ccccccccc}
\delta_V:& \cA (V) & \rightarrow & \cA_1\cA_2 (V) & &  \ve_V:& \cA (V) & \rightarrow & V\\
&\otimes & \mapsto & \otimes_1+\otimes_2 & & & \otimes &\mapsto& 0\\
&v & \mapsto &v &&& v &\mapsto &v \, .\end{array}$$

The infinitesimal distributive law reads
$$\begin{array}{cccc}
\lambda_V:& \cA_1\cA_2(V) & \rightarrow & \cA_1\cA_2(V) \\
&\otimes_1 & \mapsto & \otimes_1+\otimes_2 \\
&\otimes_2 & \mapsto & \otimes_1 \\
& v & \mapsto &v \, .
\end{array}$$
We easily see that $m$ is associative, $\delta$ is coassociative, and $\lambda$ is an entwining.

As an example, we check one of the diagrams for $\lambda$ (see \ref{f.entw}):
$$
\xymatrix{
   \cA_1\cA_2\cA_3\ar[rr]^{m \cA} \ar[d]_{\cA\lambda} & & \cA_1\cA_2 \ar[d]^\lambda \\
   \cA_1\cA_2\cA_3 \ar[r]_{\lambda \cA} & \cA_1\cA_2\cA_3 \ar[r]_{\cA m} & \cA_1\cA_2 \,.} $$
The top arrows send $\otimes_1\mapsto\otimes_1\mapsto \otimes_1+\otimes_ 2$,
$\otimes_2\mapsto\otimes_1\mapsto\otimes_1+\otimes_2$ and $\otimes_3\mapsto\otimes_2\mapsto\otimes_1$,
while the lower maps send
$\otimes_1\mapsto  \otimes_1\mapsto  \otimes_1+\otimes_2\mapsto \otimes_1+\otimes_2$,
$\otimes_2\mapsto \otimes_2+\otimes_3\mapsto \otimes_1+\otimes_3\mapsto \otimes_1+\otimes_2$ and
$\otimes_3\mapsto \otimes_2\mapsto \otimes_1\mapsto \otimes_1$, which proves commutativity of this diagram.

We have clearly $\delta=\lambda \cdot \cA e$ and $m= \cA\ve\cdot \lambda$. 
Consequently, Theorem \ref{bialgebra} holds.
Hereby we recover the Rigidity Theorem of Loday and Ronco for infinitesimal bialgebras which
says that
any infinitesimal bialgebra is freely and cofreely generated by its primitive part
(see \cite[Theorem 2.6]{LR}).
}
\end{thm}

\bigskip

{\bf Acknowledgments.} The second author gratefully acknowledges support
by the {\em German Academic Exchange Service (DAAD})
for a research stay at Heinrich Heine University D\"usseldorf and
the warm hospitality provided there.
He also wants to thank the {\em Shota Rustaveli National Science Foundation}
for the assistance with Grants  DI/12/5-103/11 and DI/18/5-113/13. The authors would like to thank
the organizers of the international conference "Homotopy and Non-Commutative Geometry", Tbislissi, 2011, where
they had the opportunity to start this project together.

\end{document}